\documentclass[a4paper,reqno]{amsart}

\usepackage{amssymb}

\usepackage{a4wide}
\usepackage{amsmath,amssymb,amsthm}

\usepackage[usenames]{color}

\newcounter{kommentar}

\renewcommand\ge{\geqslant}
\renewcommand\le{\leqslant}

\newcommand\forget[1]{}

\DeclareMathOperator{\Aut}{Aut}

\newcommand{\chr}{\mathop{\rm char}\nolimits}

\newenvironment{smatrix}{\left(\begin{smallmatrix}}{\end{smallmatrix}\right)}

\newcommand\dm[1]{\mathop{\mathcal{D}^\mathrm{b}}(#1)}

\DeclareMathOperator{\Ext}{Ext}
\renewcommand{\mod}{\mathop{\rm mod}\nolimits}

\DeclareMathOperator{\stmod}{\underline{mod}}
\DeclareMathOperator{\add}{add}
\DeclareMathOperator{\Hom}{Hom}
\DeclareMathOperator{\sthom}{\underline{Hom}}

\DeclareMathOperator{\proj}{proj}

\DeclareMathOperator{\ind}{ind}

\DeclareMathOperator{\lol}{\ell\ell}
\newcommand{\Rhom}{\mathop{\mathbf{R}\mathrm{Hom}}\nolimits}

\newtheorem{thm}{Theorem}[section]
\newtheorem{lma}[thm]{Lemma}
\newtheorem{prop}[thm]{Proposition}
\newtheorem{cor}[thm]{Corollary}

\theoremstyle{definition}

\theoremstyle{remark}
\newtheorem{rmk}[thm]{Remark}
\theoremstyle{definition}

\usepackage[all]{xy}
\xyoption{all}

\newcommand{\U}{\mathcal{U}}
\newcommand{\C}{\mathcal{C}}

\newcommand{\Z}{\mathbb{Z}}

\author{Erik Darp\"{o}}
\address[Darp\"o]{Link\"oping University, Department of Mathematics, SE-58183 Link\"oping, Sweden}

\author{Tor Kringeland}
\address[Kringeland]{Department of Mathematical Sciences,
  NTNU, 7491 Trondheim, Norway}

\title[$d$-Representation-finite symmetric algebras of finite representation type]{Classification of the $d$-representation-finite symmetric $k$-algebras of finite representation type}

\begin{document}


\begin{abstract}
  We give a complete classification of all $d$-representation-finite symmetric Nakayama algebras and of all $d$-representation-finite trivial extensions of path algebras of quivers, over an arbitrary field.
  As a consequence, we get a classification, up to Morita equivalence, of all $d$-representation-finite symmetric algebras of finite representation type over an algebraically closed field.
\end{abstract}

\maketitle

\section{Introduction}

  Let $d$ be a positive integer. An algebra $A$ is said to be \emph{$d$-representation-finite} if it has a $d$-cluster-tilting module $M$. In this case, the additive closure of $M$ furnishes a ``higher-dimensional'' analogue of the Auslander--Reiten theory of a representation-finite algebra, with $d$-almost-split sequences and a $d$-Auslander--Reiten translation relating the end terms of these sequences.
  The existence of a $d$-cluster-tilting module also has implications for the homological properties of $A$ as a whole.

Understanding and classifying $d$-representation-finite algebras are important problems in higher-dimensional Auslander--Reiten theory, introduced by Iyama \cite{iyama07a,iyama07b,iyama08}.
While many examples exist of $d$-rep\-re\-sen\-ta\-tion-finite algebras that have finite global dimension (e.g., \cite{hi11a,hi11b,jk19,vaso19}) or are self-injective \cite{di20,hj13,io13}, \emph{symmetric} algebras with a $d$-cluster-tilting module seem to be relatively rare. In this note, we give a classification of all finite-dimensional $d$-representation-finite symmetric algebras of finite representation type over an algebraically closed field.

Let $k$ be an arbitrary field. We denote by $T(A)$ the trivial extension algebra of a finite-dimensional $k$-algebra $A$ and by $kQ$ the path algebra of a quiver $Q$.
We prove the following results.

\begin{thm} \label{nakayama}
  Let $B$ be a ring-in\-de\-com\-po\-sa\-ble symmetric Nakayama algebra with $n$ simple modules and Loewy length $\lol(B) = an+1$, where $a\ge1$, and $d$ an integer greater than or equal to $2$. Then $B$ is $d$-representation-finite if and only if
  \[(a,n,d) \in \{(1,t,2t-1) \mid t\ge2 \} \cup \{(1,3,2), (1,6,2), (2,3,2)\} \,.\]
\end{thm}

\begin{thm} \label{ade}
  Let $Q$ be a finite connected acyclic quiver, and let $d$ be an integer greater than or equal to $2$. Then $T(kQ)$ is $d$-representation-finite if and only if one of the following holds.
  \begin{enumerate}
  \item The quiver $Q$ is of Dynkin type $A_3$ or $A_6$ and $d=2$;
  \item $Q$ is of Dynkin type $A_n$ and $d=2n-1$;
  \item $Q$ is of Dynkin type $D_4$ and $d=4$.
  \end{enumerate}
\end{thm}

Over an algebraically closed field, these results yield a classification of all $d$-representation-finite symmetric algebras of finite representation type.

\begin{cor} \label{rfsymm}
  Let $k$ be an 
  algebraically closed field, $A$ a finite-dimensional ring-in\-de\-com\-po\-sa\-ble symmetric $k$-algebra of finite representation type, and $d\ge2$ an integer. Then $A$ is $d$-representation-finite if and only if it is either
  \begin{enumerate}
  \item stably equivalent to a Nakayama algebra with $n$ simple modules and Loewy length $an+1$, where $(a,n,d)$ is as in Theorem~\ref{nakayama}; or
    \label{rfsymm1}
  \item isomorphic to the trivial extension $T(A)$ of a tilted algebra $A$ of Dynkin type $D_4$ and $d=4$. \label{rfsymm2}
  \end{enumerate}
\end{cor}

\begin{rmk}
  The case \eqref{rfsymm1} in Corollary~\ref{rfsymm} is equivalent to $A$ being Morita equivalent to the Brauer algebra of a Brauer tree with $n$ edges and an exceptional vertex of multiplicity $a$
(\cite[Theorem~2]{gr79}, \cite[Theorem~4.2]{rickard89b}).
Also, by \cite[Corollary~2.2]{asashiba99}, the condition \emph{stably equivalent} in \eqref{rfsymm1} could be replaced by \emph{derived equivalent}.
\end{rmk}

Any ($1$-)representation-finite block of a finite group algebra is stably equivalent to a symmetric Nakayama algebra. Thus, in particular, Theorem~\ref{nakayama} entails the classification of all $d$-representation-finite block algebras of finite representation type.
On the other hand, Erdmann and Holm \cite{eh08} have shown that the complexity of a $d$-representation-finite self-injective algebra is at most one. A consequence of this (see \cite[5.3]{eh08}) is that if $A$ is a $d$-representation-finite block algebra of infinite representation type, then $d=3$, $\chr k = 2$ and the defect group of $A$ is generalised quaternion. A first example of a $3$-representation-finite block algebra of this type has recently been found by B\"ohmler and Marczinzik \cite{bm22}.

  We remark that every representation-finite self-injective algebra $A$ is \emph{periodic}, in the sense that it has a periodic minimal projective resolution as an $A$-$A$-bimodule \cite{dugas10}.
  For $d\ge2$, it is an open question whether $d$-representation-finiteness implies periodicity for self-injective algebras. This question is related to the so-called \emph{periodicity conjecture}, introduced by Erdmann and Skowr\'onski in \cite{es15}. The relation between cluster tilting, periodicity and the fractional Calabi--Yau property is studied in \cite{cdim25,jm23,jm25}, and there is ongoing work towards a classification of all tame symmetric algebras of period four \cite{ehs24,es18,es19,es20,hss23}.

\section{Preliminaries}

Throughout this note, $k$ is a field, $A$ a finite-dimensional $k$-algebra and $d$ an integer greater than or equal to $2$.
We denote by $\mod A$ the category of finite-dimensional right $A$-modules. The \emph{stable module category}, $\stmod A$, is the quotient of $\mod A$ by the ideal of morphisms factoring through a projective module.
When $A$ is a $\mathbb{Z}$-graded algebra, $\mathop{\mathrm{mod}}\nolimits ^{\mathbb{Z}}A$ denotes
the category of graded finite-dimensional $A$-modules, and $\stmod ^{\mathbb{Z}}A$ denotes the corresponding stable category.
If $A$ is self-injective, then $\stmod A$ has the structure of a triangulated category, with suspension functor given by the co-syzygy functor: $\Omega^{-1}:\stmod A \to \stmod A$. 

A \emph{$d$-cluster-tilting subcategory} of an abelian or triangulated category $\C$ is a functorially finite full subcategory $\U\subset\C$ satisfying
\begin{align*}
  \U &= \left\{ X\in\C \mid \forall\, U\in\U,\:i\in\{1,\ldots,d-1\}: \Ext^i(X,U) = 0\right\} \\
  &= \left\{ X\in\C \mid \forall\, U\in\U,\:i\in\{1,\ldots,d-1\}: \Ext^i(U,X) = 0\right\} \,.
\end{align*}
A \emph{$d$-cluster-tilting module} is an $A$-module $M$ for which the additive closure, $\add M$, forms a $d$-cluster-tilting subcategory of $\mod A$.
An algebra $A$ is \emph{$d$-representation-finite} if it has a $d$-cluster-tilting module. 
Note that $\add M\subset\mod A$ is functorially finite for any $M\in\mod A$.

Let $A$ be self-injective and let $U$ be a finite-dimensional $A$-module. Then the subcategory $\add U\subset \stmod A$ is $d$-cluster-tilting if and only if $U\oplus A\oplus DA$ is a $d$-cluster-tilting $A$-module. In particular, $d$-representation-finiteness is preserved by stable equivalence: if $A'$ is self-injective and there is a triangle equivalence $\stmod A\simeq \stmod A'$, then $A$ is $d$-representation-finite if and only if so is $A'$.

We denote by $D$ the standard $k$-duality: $D=\Hom_k(-,k):\mod A\to \mod A^{\rm op}$.
The \emph{trivial extension} of $A$ is the algebra
\[T(A) = A\oplus DA\qquad \mbox{with multiplication} \qquad (a,f)(b,g) = (ab, ag + fb),\]
where the products $ag$ and $fb$ are given by the bimodule structure of $DA$. 
Trivial extension algebras are \emph{symmetric}, meaning that $T(A)$ and $DT(A)$ are isomorphic as $T(A)$-bimodules.
Setting $T(A)_0 = A$ and $T(A)_1 = DA$ gives a $\mathbb{Z}$-grading on $T(A)$.
Let $\C$ be a $k$-linear triangulated category with $k$-dual $D:\C\to\C^{\rm op}$. 
A \emph{Serre functor} on $\C$ is an auto-equivalence $\mathbb{S}:\C\to\C$ satisfying a bifunctorial isomorphism
$\C(X,\mathbb{S}(Y))\simeq D\C(Y,X)$.
For any symmetric algebra $A$, the syzygy functor $\Omega$ is a Serre functor on $\stmod A$.
If $A$ has finite global dimension, then the \emph{Nakayama functor}
\[\nu = D\circ\Rhom_{A}(-,A)\simeq \,- \otimes^{\mathbf{L}}_{A}DA:\dm{A}\to\dm{A} \] 
is a Serre functor on the bounded derived category $\dm{A}$ of finite-dimensional $A$-modules.
In this case, Happel's theorem \cite{happel88} gives a triangle equivalence
\[\dm{A} \simeq \stmod^\Z T(A),\]
under which the grade shift functor $(1)$ on $\stmod^\Z T(A)$ corresponds to the auto-equivalence $\nu\circ[1]$ of $\dm{A}$.
Note that $\nu$ is a triangle functor on $\dm{A}$, and thus it commutes with $[1]$ up to natural isomorphism. 

For a quiver $Q$, arrows in the path algebra $kQ$ are composed from left to right. Consequently, modules of $kQ$ correspond to covariant representations of the quiver.

Let $Q$ be a quiver of simply laced Dynkin type. Then the path algebra $kQ$ \cite{gabriel72}, as well as its trivial extension $T(kQ)$ \cite{tachikawa80}, is representation finite.
The forgetful functor $\mod^\Z T(kQ)\to \mod T(kQ)$ is dense (that is, every $T(kQ)$-module is gradable) \cite{gg82b} and there are triangle equivalences
\[\stmod T(kQ) \simeq (\stmod^\Z T(kQ))/(1) \simeq \dm{kQ}/(\nu\circ[1]) \,.\]
It follows that basic $d$-cluster-tilting modules of $T(kQ)$ correspond bijectively to $(\nu\circ[1])$-equivariant $d$-cluster-tilting subcategories of $\dm{kQ}$.

  Let $\Z Q$ be the quiver with vertex set $(\Z Q)_0 = Q_0\times \Z$ and, for each arrow $a:i\to j$ in $Q$ and $l\in\Z$, two arrows $a = a_l:(i,l)\to (j,l)$ and $a^*=a^*_l: (j,l)\to (i,l+1)$ in $\Z Q$. The $k$-linear \emph{path category} of $\Z Q$ has object set $(\Z Q)_0$ and morphism spaces $(i,l)\to(j,m)$ spanned by all paths from $(i,l)$ to $(j,m)$. The \emph{mesh category} $k(\Z Q)$ is the quotient of the path category by the ideal generated by all elements of the form
  \[\sum_{a\in Q_1 ,\, s(a) = i}a_la_l^* + \sum_{b\in Q_1 ,\, t(b) = i}b_l^*b_{l+1}\in\Hom((i,l),(i,l+1))\]
where $(i,l)\in (\Z Q)_0$.
The full subcategory of $\dm{kQ}$ formed by all indecomposable objects is equivalent to
$k(\Z Q)$.
The Auslander--Reiten translation on $\dm{kQ}$ is $\tau = \nu\circ[-1]$, and corresponds to the shift given by $l\mapsto l-1$ in $k(\Z Q)$. In particular, every indecomposable object in $\dm{kQ}$ is isomorphic to $\tau^l(P)$ for a unique indecomposable projective $kQ$-module $P$ and an integer $l\in\Z$. \cite[Section~I.5]{happel88}

  The following result will be needed for the proof of Corollary~\ref{rfsymm}. While it is known to experts in the field (see, for example, the proofs of \cite[Lemma~4.12]{ckl15}, \cite[Theorem~6.1]{dugas10}, \cite[Proposition~4.2]{es06}), we have not been able to find a proof of it in the literature.
Let $\ind A$ be the full subcategory of $\mod A$ formed by a chosen set of representatives for the isomorphism classes of finite-dimensional indecomposable $A$-modules. 
For unexplained terminology, we refer to \cite{bg81}.

\begin{prop} \label{omega-phi}
  Assume that $k$ is algebraically closed. Let $A$ and $A'$ be representation-finite self-injective $k$-algebras with Auslander--Reiten quivers $\Gamma_A$ and $\Gamma_{A'}$, respectively, and $\phi:\Gamma_A\to\Gamma_{A'}$ an isomorphism of translation quivers. 
  Then the bijection $\Phi:\ind A\to \ind A'$ induced by $\phi$ satisfies
  $\Omega_{A'}(\Phi(X)) \simeq \Phi(\Omega_A(X))$ for all non-projective $X\in\ind A$.
\end{prop}

\begin{proof}
We will prove the result in the more general setting of locally representation-finite $k$-categories. Thus, let $\mathcal{A}$ and $\mathcal{A}'$ be locally representation-finite self-injective $k$-categories, with Auslander--Reiten quivers $\Gamma_{\mathcal{A}}$ and $\Gamma_{\mathcal{A}'}$, respectively. Moreover, without loss of generality, we may assume that $\mathcal{A}$ and $\mathcal{A}'$ are connected.
In the setup of the proposition, we take $\mathcal{A} = \proj A$ and $\mathcal{A}'=\proj A'$.

Let $\pi:\tilde{\Gamma}\to\Gamma_{\mathcal{A}}$ be the universal cover of $\Gamma_{\mathcal{A}}$. By \cite[Example~3.1(b)]{bg81}, there exists a covering functor $F_{\mathcal{A}}:k(\tilde{\Gamma})\to\ind\mathcal{A}$ which agrees with $\pi$ on indecomposables. Let $\mathcal{P}$ and $\mathcal{P}_{\mathcal{A}}$ denote the full subcategories of projective objects in $k(\tilde{\Gamma})$ and $\ind \mathcal{A}$, respectively, and $p_\mathcal{A}:\mathcal{P}\to\mathcal{P}_{\mathcal{A}}$ the covering induced by $F_{\mathcal{A}}$.

For $x\in k(\tilde{\Gamma})$ and $y\in\ind\mathcal{A}$, let $M_x = \Hom_{k(\tilde{\Gamma})}(-,x)|_{\mathcal{P}}\in \mod\mathcal{P}$ and $M^{\mathcal{A}}_y = \Hom_{\mathcal{A}}(-,y)|_{\mathcal{P}_{\mathcal{A}}}\in\mod\mathcal{P}_{\mathcal{A}}$. 
By \cite[Proposition~2.4]{bg81}, the assignments $x\mapsto M_x$ and $y\mapsto M^{\mathcal{A}}_y$ give rise to isomorphisms of categories $k(\tilde{\Gamma}) \stackrel{\sim}{\to} \ind\mathcal{P}$ and $\ind\mathcal{A} \stackrel{\sim}{\to} \ind\mathcal{P}_{\mathcal{A}}$, respectively.
The push-down functor $p_{\mathcal{A}\,\bullet}:\mod\mathcal{P}\to\mod\mathcal{P}_\mathcal{A}$ is exact,
$p_{\mathcal{A}\,\bullet}(M_x) = M^{\mathcal{A}}_{F_{\mathcal{A}}(x)}$,
and the projective objects in $\mod \mathcal{P}_{\mathcal{A}}$ are precisely the images under $p_{\mathcal{A}\bullet}$ of the projectives in $\mod\mathcal{P}$ \cite[3.2]{bg81}.
The category $\mod\mathcal{P}$ being Frobenius, $\Omega_{\mathcal{P}}:\stmod\mathcal{P} \to \stmod\mathcal{P}$ is an auto-equivalence, inducing a permutation $\omega$ of the non-projectives in $k(\tilde{\Gamma})$
such that $\Omega_{\mathcal{P}}(M_x)\simeq M_{\omega(x)}$. Thus there is a short exact sequence
\[
  0\to M_{\omega(x)}\to P \to M_x\to 0
\]
in $\mod\mathcal{P}$, with $P$ projective.
Applying the push-down functor $p_{\mathcal{A}\bullet}$ to this sequence yields an exact sequence
\[
  0\to M^{\mathcal{A}}_{F_{\mathcal{A}}(\omega(x))} \to p_{\mathcal{A}\bullet}(P) \to M^{\mathcal{A}}_{F_{\mathcal{A}}(x)}\to 0
\]
in $\mod\mathcal{P}_{\mathcal{A}}$, whence it follows that
$\Omega_{\mathcal{P}_\mathcal{A}}(M^{\mathcal{A}}_{F_{\mathcal{A}}(x)})\simeq M^{\mathcal{A}}_{F_{\mathcal{A}}(\omega(x))}$.
Translated through the isomorphism $\ind\mathcal{A} \stackrel{\sim}{\to}\ind\mathcal{P}_{\mathcal{A}}$, this means that
$\Omega_{\mathcal{A}}(F_{\mathcal{A}}(x))\simeq F_{\mathcal{A}}(\omega(x))$ in $\mod\mathcal{A}$.

Let $\phi:\Gamma_{\mathcal{A}} \to \Gamma_{\mathcal{A}'}$ be an isomorphism of translation quivers.
Then $\pi' = \phi\pi:\tilde{\Gamma}\to \Gamma_{\mathcal{A}'}$ is the universal cover of $\Gamma_{\mathcal{A}'}$, and there is a covering functor $F_{\mathcal{A}'}: k(\tilde{\Gamma}) \to \ind\mathcal{A}'$ such that $F_{\mathcal{A}'}(x)\in\ind\mathcal{A}'$ corresponds to the vertex $\phi(\pi(x))$ in $\Gamma_{\mathcal{A}'}$.
Proceeding as for $\mathcal{A}$ above, we get an isomorphism 
$\Omega_{\mathcal{A}'}(F_{\mathcal{A}'}(x))\simeq F_{\mathcal{A}'}(\omega(x))$ in $\mod\mathcal{A}'$.

The bijection $\Phi:\ind\mathcal{A}\to\ind\mathcal{A}'$ induced by $\phi$ sends $F_{\mathcal{A}}(x)$ to $F_{\mathcal{A}'}(x)$. Thus, for an arbitrary non-projective $X = F_{\mathcal{A}}(x)\in \ind\mathcal{A}$, 
\begin{align*}
  \Omega_{\mathcal{A}'}(\Phi(X)) &= \Omega_{\mathcal{A}'}(\Phi(F_{\mathcal{A}}(x))) = \Omega_{\mathcal{A}'}(F_{\mathcal{A}'}(x)) \simeq F_{\mathcal{A}'}(\omega(x)) \\
                                 &\simeq \Phi(F_{\mathcal{A}}(\omega(x))) \simeq \Phi(\Omega_{\mathcal{A}}(F_{\mathcal{A}}(x))) = \Phi(\Omega_{\mathcal{A}}(X)). \qedhere
\end{align*}
\end{proof}

\begin{prop}[{\cite[Proposition~3.4]{iy08}}] \label{nud}
  Let $\C$ be a $k$-linear triangulated category with a Serre functor $\mathbb{S}$.
  Then every $d$-cluster-tilting subcategory $\U$ of $\C$ satisfies $(\mathbb{S}\circ[-d])(\U)=\U$.
\end{prop}

For $A$ of finite global dimension, we write
\[\nu_d = \nu\circ[-d] :\dm{A}\to\dm{A} .\]

Let $l,m\in\Z$, $l>0$. An algebra $A$ of finite global dimension is \emph{twisted $(m/l)$-Calabi--Yau} if there exists a $k$-linear automorphism $\phi\in\Aut_k(A)$ such that $\nu^l\simeq [m]\circ\phi^*$,
where $\phi^*$ denotes the auto-functor $\dm{A}\to\dm{A}$ induced by $\phi$.
If $\phi = \mathop{\rm id}$, then $A$ is $(m/l)$-Calabi--Yau. The algebra $A$ is \emph{(twisted) fractionally Calabi--Yau} if it is (twisted) $(m/l)$-Calabi--Yau for some $l$ and $m$. 
We remark that while, for an $(m/l)$-Calabi--Yau algebra $A$,
 $m/l$ is uniquely determined by $A$ as a rational number, being $(am/al)$--Calabi--Yau does not in general imply being $(m/l)$-Calabi--Yau for $a\ge2$.

\begin{prop}[{\cite[0.3]{my01}, see also \cite[Theorem~8.1]{cdim25}}]
  \label{cydim}
  Let $Q$ be a quiver of simply laced Dynkin type. The path algebra $kQ$ is $(h-2)/h$-Calabi--Yau, where $h$ denotes the Coxeter number of the Dynkin diagram (Table~\ref{hnumbers}).
If $Q$ is of type $A_1$, $D_{2m}$, $E_7$ or $E_8$, then $kQ$ is $\left(\frac{h}{2}-1\right)/\left(\frac{h}{2}\right)$-Calabi--Yau.
\end{prop}

\begin{table}[htb]
  \[
  \begin{array}{|c||c|c|c|c|c|}
    \hline
    Q & A_n & D_n & E_6 & E_7 & E_8\\ \hline
    h & n+1 & 2(n-1) & 12 & 18 & 30\\ \hline
  \end{array}
  \]
  \caption{Coxeter numbers of simply laced Dynkin diagrams.}
  \label{hnumbers}
\end{table}

\begin{lma} \label{isom}
  The following isomorphisms of triangle functors $\dm{kQ}\to\dm{kQ}$ hold:
  \begin{align}
    [h-2] &\simeq \nu^h, \label{isom1} \\
    \nu\circ[1] &\simeq \tau^{1-h}, \label{isom2} \\
    [2] &\simeq  \tau^{-h} \simeq (\nu\circ[1])\circ\tau^{-1}. \label{isom3}
  \end{align}
\end{lma}

\begin{proof}
The existence of an isomorphism $[h-2]\simeq\nu^h$ of triangle functors follows from Proposition~\ref{cydim} and \cite[Proposition~2.7(b)]{cdim25}.
The remaining isomorphisms are straightforward consequences of the first one, together with the identity $\tau = \nu\circ [-1]$.
\end{proof}

\begin{lma} \label{periodicity}
 If $T(kQ)$ is $d$-representation-finite, then $(d+1)\mid 2(h-1)$.
\end{lma}

\begin{proof}
Let $\U$ be a $(\nu\circ[1])$-equivariant $d$-cluster-tilting subcategory of $\dm{kQ}$. 
By Lemma~\ref{isom}\eqref{isom1}, there are isomorphisms
\[ [2(h-1)] = [h-2]\circ[h] \simeq \nu^h\circ [h] \simeq (\nu\circ[1])^h \]
of triangle functors $\dm{kQ}\to\dm{kQ}$ and thus $\U[2(h-1)] = (\nu\circ[1])^h(\U) = \U$.
On the other hand,
\[ \U[d+1] = \nu_d(\U[d+1]) = (\nu\circ[1])(\U) = \U, \]
and it follows that $\U[g] = \U$, where $g=\gcd(d+1,2(h-1))$.

For all $X,Y\in\U$ and $r\in\{1,\ldots,d-1\}$, we have $\Hom_{\dm{kQ}}(X,Y[r]) = \Ext^r(X,Y) = 0$. Hence $\U[r] = \U$ is impossible for any $r\in\{1,\ldots,d-1\}$ and, consequently, $g\ge d$. With $g\mid (d+1)$ and $d\ge2$, this implies that $g=d+1$, and thus $(d+1)\mid 2(h-1)$.
\end{proof}

\section{Proof of Theorem~\ref{nakayama}}

Theorem~\ref{nakayama} is an application of \cite[Section~5]{di20}, in which a characterisation of all $d$-representation-finite self-injective Nakayama algebras was given.
Throughout this section, let $B$ be a ring-indecomposable symmetric Nakayama $k$-algebra with $n$ simple modules and Loewy length $\lol(B) = an+1$.
Specialising \cite[Theorem~5.1]{di20} to this situation, we get the following result.
As usual, $d$ is an integer greater than or equal to $2$.

\begin{prop} \label{siny}
  The algebra $B$ is $d$-representation-finite if and only if at least one of the following conditions holds:
  \begin{itemize}
  \item[(a)] $((an+1)(d-1) +2)\mid 2n$;
  \item[(b)] $((an+1)(d-1) +2)\mid tn $,\quad where\quad $t = \gcd(d+1, 2an)$.
  \end{itemize}
\end{prop}

\begin{lma} \label{conditiona}
  The condition (a) in Proposition~\ref{siny} holds if and only if $(a,n,d)=(1,3,2)$.
\end{lma}

\begin{proof}
  As $(an+1)(d-1) +2>an$, the condition (a) is equivalent to $a=1$ and $(n+1)(d-1)+2=2n$. The last equation can be rewritten as $(n+1)(3-d) =4$. Since $d\ge2$, this holds if and only if $n=3$ and $d=2$.
\end{proof}

\begin{lma} \label{divides2n}
  If $B$ is $d$-representation-finite, then $(d+1)\mid 2n$. 
\end{lma}

In particular, the condition (b) in Proposition~\ref{siny} is equivalent to 
\begin{equation} \label{newb}
  ((an+1)(d-1)+2)\mid (d+1)n.
\end{equation}

\begin{proof}
Since $B$ is a ring-indecomposable symmetric Nakayama algebra, the Auslander--Reiten translate $\tau\simeq \Omega^2:\stmod B\to\stmod B$ cyclically permutes the simple modules \cite[Corollary~IV.2.12]{ars95}. 
Hence $\Omega^{2n}(S)\simeq S$ for all simple modules $S\in\mod B$.
Moreover, $\lol(\Omega^2(M))=\lol(M)$ for all $M\in\mod B$ \cite[Corollary~IV.2.9]{ars95}, and indecomposable $B$-modules are classified by their tops and Loewy lengths \cite[Lemma~IV.2.5]{ars95}, so it follows that $\Omega^{2n}(M)\simeq M$ for all $M\in\stmod B$.

Let $M$ be a $d$-cluster-tilting $B$-module. Then $\Omega^{d+1}(M) = \mathbb{S}_d(M)\simeq M$ in $\stmod B$ which, combined with $\Omega^{2n}(M)\simeq M$, yields $\Omega^{\gcd(d+1,2n)}(M)\simeq M$. Since $\sthom(\Omega^r(M),M) \simeq \Ext^r(M,M)=0$ for all $r\in\{1,\ldots,d-1\}$, this implies that $\gcd(d+1,2n)\ge d$ and, as $d\ge2$, it follows that $(d+1)\mid 2n$. 
\end{proof}

\begin{lma} \label{a=2}
  If the condition (b) in Proposition~\ref{siny} holds, then either $(a,n,d) = (2,3,2)$ or $a=1$. 
\end{lma}

\begin{proof}
If (b) holds, then
\[((an+1)(d-1)+2) \le tn = (d+1)n,\]
which is equivalent to
\[(a-1)(d-1)n + d+1\le 2n .\]
Thus $(a-1)(d-1)\in\{0,1\}$, implying that either $a=1$ or $a=d=2$. In the latter case, it follows from \eqref{newb} that $n=d+1=3$.
\end{proof}

Summarising the results above, we see that $B$ is $d$-representation-finite if and only if either $(a,n,d)\in\{(1,3,2), (2,3,2)\}$ or $a=1$, $(d+1)\mid 2n$ and
\begin{equation} \label{newconditionb}
  ((n+1)(d-1) + 2)\mid (d+1)n.
\end{equation}
To prove Theorem~\ref{nakayama}, we need to show that, in the latter case, either $d=2n-1$ or $(d,n) = (2,6)$ holds.

By Lemma~\ref{divides2n}, there exists an integer $b\ge1$ such that $2n = b(d+1)$. Multiplying both sides of \eqref{newconditionb} by $2$ gives
\[ 2((n+1)(d-1) +2) \mid 2n(d+1) \,. \]
Since
\begin{align*}
  2((n+1)(d-1) +2) &= 2n(d-1) + 2(d+1) = b(d+1)(d-1) +2(d+1) \\
  &= (b(d-1)+2)(d+1) \quad \mbox{and} \\ 
  2n(d+1) &= b(d+1)^2 \,,
\end{align*}
this is equivalent to
\begin{equation}
  (b(d-1)+2) \mid b(d+1). \label{divides}
\end{equation}
Let $c\ge1$ be such that $b(d+1) = c(b(d-1)+2)$.
Then $(c-1)(d-1)b = 2(b-c)<2b$, whence $(c-1)(d-1)<2$, so either $c=1$, or $c=d=2$.
Thus \eqref{divides} holds if and only if either $b(d-1)+2 = b(d+1)$, or $d=2$ and $3b=2(b+2)$. In the former case $b=1$ and thus $2n = b(d+1) = d+1$.
In the latter, $b=4$ and $2n = 4(2+1) = 12$, so $(d,n) = (2,6)$.
This concludes the proof of Theorem~\ref{nakayama}.

\section{Proof of Theorem~\ref{ade}}

To start with, observe that if $Q$ is a quiver of Dynkin type $A_n$, then $T(kQ)$ is stably equivalent to a symmetric Nakayama algebra with $n$ simple modules and Loewy length $n+1$.
By Theorem~\ref{nakayama}, such an algebra is $d$-representation-finite if and only if
either $d=2n-1$, or $d=2$ and $n\in\{3,6\}$. Since $d$-representation-finiteness is preserved by stable equivalence, this establishes the claim in Theorem~\ref{ade} for Dynkin quivers of type $A$. 

\begin{prop}
Let $Q$ be an acyclic quiver. If $T(kQ)$ is $d$-representation-finite for some $d$, then $Q$ is a simply laced Dynkin quiver.
\end{prop}

\begin{proof}
  By \cite[Theorem~1.3]{cdim25}, if $T(kQ)$ is $d$-representation-finite then $kQ$ is twisted fractionally Calabi--Yau.
This means that  $\nu^l\simeq [m]\circ \phi^*$ for some $l,m\in\Z$, $l>0$ and $\phi\in\Aut_k(kQ)$, and thus $\nu_1^{-l}(kQ)\simeq (\phi^*)^{-1}(kQ)[l-m]\simeq kQ[l-m]$.

On the other hand, if $Q$ is not Dynkin, then $kQ$ is a representation-infinite hereditary algebra.
Consequently, for all $l>0$, the object $\nu_1^{-l}(kQ)=\tau^{-l}(kQ)$ in $\dm{kQ}$ is a non-projective stalk complex concentrated in degree zero \cite[Proposition~VIII.1.15]{ars95}.
So in this case, $\nu_1^{-l}(kQ)$ is not isomorphic to $kQ[l-m]$ for any $m\in\Z$, and thus $kQ$ cannot be twisted fractionally Calabi--Yau.
\end{proof}

It remains to consider quivers of Dynkin types $D$ and $E$.
For a vertex  $i$ of the quiver $Q$, we denote by $P_i$ the projective cover of the simple $kQ$-module supported at $i$, and by
$\mathcal{O}_i = \add\{\tau^l(P_i) \mid l\in\Z\}\subset\dm{kQ}$ the $\tau$-orbit of $P_i$ in $\dm{kQ}$. Recall that, since $Q$ is Dynkin, every indecomposable object in $\dm{kQ}$ is contained in a unique orbit $\mathcal{O}_i$.

\begin{figure}[htb]
  \centerline{
    \xymatrix@=1em{
      1\ar@{-}[r] & 2\ar@{-}[r] & 3\ar@{-}[r]\ar@{-}[d] & \cdots \ar@{-}[r] & m \\
      & & 4 & &
    }
  }

  \caption{The Dynkin diagram $\mathbb{E}_m$, $m\in\{6,7,8\}$.}
  \label{Em}
\end{figure}
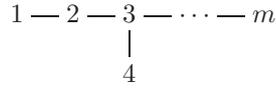

\begin{prop}
  Let $Q$ be a quiver of Dynkin type $E$. Then $T(kQ)$ is not $d$-representation-finite for any $d\ge2$. 
\end{prop}

\begin{proof}
Let $Q$ be a quiver of Dynkin type $E_m$, $m\in\{6,7,8\}$, with vertices enumerated as in Figure~\ref{Em}, and arrows oriented in the direction of decreasing index. For $i\in\{2,\ldots,m\}$, let $a_i$ be the unique arrow in $Q$ starting in $i$.
Then $a_{4,l}a_{3,l}a_{3,l}^*a_{4,l+1}^* \in \Hom_{k(\Z Q)}( (4,l), (4,l+2) )$ is non-zero for any $l\in\Z$. 
Using the equivalence between $k(\Z Q)$ and the category of indecomposable objects in $\dm{kQ}$, we now get, for an arbitrary object $X = \tau^{-l}(P_4)\in\mathcal{O}_4$, that
\[
  \Hom_{\dm{kQ}}((\nu\circ[1])^2(X),X[4]) \simeq \Hom_{\dm{kQ}}(X,\tau^{-2}(X)) \simeq
  \Hom_{k(\Z Q)}( (4,l), (4,l+2) ) \ne 0 \,.
\]
Similarly, since the morphisms
\begin{align*}
  a_2^*a_3^*a_4^*a_4a_3a_2 &\in\Hom_{k(\Z Q)}( (1,l), (1,l+3) ), \\
  a_m\cdots a_5a_3a_3^*a_5^*\cdots a_m^* &\in\Hom_{k(\Z Q)}((m,l),(m,l+m-3)),\\
  a_ia_i^* &\in\Hom_{k(\Z Q)}((i,l),(i, l+1)) \quad\text{for $i\ne 1, 4, m$},
\end{align*}
 in $k(\Z Q)$ are non-zero, we get
  \begin{align*}
    &\Hom_{\dm{kQ}}((\nu\circ[1])^3(X),X[6]) \simeq \Hom_{\dm{kQ}}(X,\tau^{-3}(X)) \ne 0, &&\mbox{if}\; X\in\mathcal{O}_1; \\
    &\Hom_{\dm{kQ}}((\nu\circ[1])^{m-3}(X),X[2(m-3)]) \simeq \Hom_{\dm{kQ}}(X,\tau^{3-m}(X)) \ne 0, &&\mbox{if}\; X\in\mathcal{O}_m; \\
    &\Hom_{\dm{kQ}}((\nu\circ[1])(X), X[2]) \simeq \Hom_{\dm{kQ}}(X,\tau^{-1}(X)) \ne 0, &&\mbox{if}\; X\in\mathcal{O}_i,\: i\ne 1,4,m.
  \end{align*}
Thus, for any $(\nu\circ[1])$-equivariant subcategory $\U$ of $\dm{kQ}$, we have $\Hom_{\dm{kQ}}(\U,\U[r])\ne0$ for some $r\in\{1,\ldots, 2(m-3)\}$.
Observe that $2(m-3)<h-2$ for each $m=6,7,8$\forget{$m\in\{6,7,8\}$}, where $h=12,18,30$\forget{$h\in\{12,18,30\}$} is the Coxeter number of $E_m$. Hence, if $\dm{kQ}$ has a $(\nu\circ[1])$-equivariant $d$-cluster-tilting subcategory then $d<h-2$.

On the other hand, by Lemma~\ref{periodicity}, the existence of such a subcategory implies that $(d+1)\mid 2(h-1)$. Since $h-1\in\{11,17,29\}$ is a prime number and $d+1<h-1$, it follows that $d+1=2$ and thus $d=1$.
So $\dm{kQ}$ does not have a $(\nu\circ[1])$-equivariant $d$-cluster-tilting subcategory for any $d\ge2$, and hence $T(kQ)$ is not $d$-representation-finite.

The trivial extension algebra $T(kQ')$ of an arbitrary quiver $Q'$ of Dynkin type $E$ is stably equivalent to $T(kQ)$, and thus not $d$-representation-finite.
\end{proof}

For the remainder of this section, let $Q$ be a quiver of Dynkin type $D_n$, with $n\ge4$ and vertices enumerated as in Figure~\ref{Dn}.

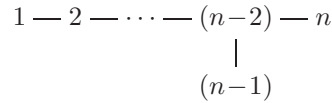
\begin{figure}[htb]
  \centerline{
    \xymatrix@=1em{
      1\ar@{-}[r] & 2\ar@{-}[r] & \cdots\ar@{-}[r] & (n\!-\!2)\ar@{-}[r]\ar@{-}[d] & n \\
      & & & (n\!-\!1)
    }
  }
  \caption{The Dynkin diagram $\mathbb{D}_n$.}
  \label{Dn}
\end{figure}

Grimeland \cite{grimeland16} has studied invariance properties of $2$-cluster-tilting subcategories of $\dm{A}$ for representation-finite hereditary algebras $A$. 
The following result is a direct consequence of \cite[Corollary~44]{grimeland16}.

\begin{prop} \label{dnot2}
The category $\dm{kQ}$ does not have a $(\nu\circ[1])$-equivariant $2$-cluster-tilting subcategory; hence, $T(kQ)$ is not $2$-representation-finite.
\end{prop}

Now assume that $\dm{kQ}$ has a $(\nu\circ[1])$-equivariant $d$-cluster-tilting subcategory $\U$ for some $d\ge3$.
Since the Coxeter number of $D_n$ is $h=2n-2$, Lemma~\ref{periodicity} tells us that
\begin{equation} \label{Dndiv}
  (d+1)\mid 2(2n-3).
\end{equation}
This immediately implies that $d\ne3$, since $2(2n-3)=4n-6$ is not divisible by $4$. 

\begin{lma} \label{ontherim}
  The $d$-cluster-tilting subcategory $\U$ of $\dm{kQ}$ is contained in $\add\left(\mathcal{O}_1\cup\mathcal{O}_{n-1}\cup\mathcal{O}_n\right)$, but not in $\mathcal{O}_1$. 
\end{lma}

\begin{proof}
  Let $X\in\U$ be indecomposable. Since $\U$ is $(\nu\circ[1])$-equivariant and $d\ge4$, by Lemma~\ref{isom}\eqref{isom3} we get that 
  \[\Hom_{\dm{kQ}}(X,\tau^{-1}X)\simeq \Hom_{\dm{kQ}}((\nu\circ[1])(X),X[2]) = 0. \]
If $l\in \Z$ and $1< i < n-1$ then either $a_ia_i^*$ or $a_i^*a_i$ (depending on the orientation of $Q$) is a non-zero morphism $(i,l)\to (i,l+1)$ in $k(\Z Q)$. Thus
$X\in\mathcal{O}_1\cup\mathcal{O}_{n-1}\cup\mathcal{O}_n$.

Assume now that $\U\subset\mathcal{O}_1$, and hence that $X\in\mathcal{O}_1$. 
From the shape of $\Z Q$, one reads off that there exists a unique (up to isomorphism) indecomposable $\widetilde{X}\in \mathcal{O}_n$ such that $\Hom_{\dm{kQ}}(X,\widetilde{X})\ne 0$. Moreover, if $f:X\to \widetilde{X}$ is a non-zero morphism and $Y\in\mathcal{O}_1$ an indecomposable object not isomorphic to $X$, then every morphism $X\to Y$ factors uniquely through $f$. Hence, $f$ induces an isomorphism $f^*:\Hom_{\dm{kQ}}(\widetilde{X},Y)\stackrel{\sim}{\to} \Hom_{\dm{kQ}}(X,Y),\: g\mapsto gf$.

Note that $\mathcal{O}_1[1] = \mathcal{O}_1$. Now, for any indecomposable $Y\in\U\subset\mathcal{O}_1$ and any $r\in\{1,\ldots,d-1\}$, we have $\Hom_{\dm{kQ}}(X,Y[r])=0$, implying that $Y[r]\not\simeq X$, and thus \[\Hom_{\dm{kQ}}(\widetilde{X},Y[r])\simeq \Hom_{\dm{kQ}}(X,Y[r])=0.\]
This proves that $\widetilde{X} \in\{Z\in\dm{kQ}\mid \Ext^i(Z,\U)=0,\:\forall i\in\{1,\ldots,d-1\}\}=\U$; hence, $\U\subset\mathcal{O}_1$ is impossible.
\end{proof}

With the following proposition, we conclude the proof of the necessity part of Theorem~\ref{ade}.

\begin{prop}
  The algebra $T(kQ)$ is $d$-representation-finite only if $d=n=4$. 
\end{prop}

\begin{proof}
  As before, $\U$ denotes a $(\nu\circ[1])$-equivariant $d$-cluster-tilting subcategory of $\dm{kQ}$. By Lemma~\ref{ontherim}, there exists some indecomposable object $X\in\U$ that belongs to either $\mathcal{O}_{n-1}$ or $\mathcal{O}_n$.
As $[4]\simeq \tau^{-2}\circ(\nu\circ[1])^2$ by Lemma~\ref{isom}\eqref{isom3}, we get
  \[0\ne\Hom_{\dm{kQ}}(X,\tau^{-2}(X)) \simeq \Hom_{\dm{kQ}}((\nu\circ[1])^2(X),X[4])\]
  and, since $X$ and $(\nu\circ[1])^2(X)$ both belong to $\U$, this implies that $d\le4$. On the other hand, from Proposition~\ref{dnot2} and \eqref{Dndiv} we know that $d>3$, hence $d=4$.

Note that $\Hom_{\dm{kQ}}(\tau^{2r}(X),X)\ne0$ if and only if $0\le 2r\le n-2$. 
From the definitions, we get the following isomorphisms of triangle functors on $\dm{kQ}$:
 \[(\nu\circ[1])^3\circ\nu_4\circ[-3] \simeq \nu^4\circ[3-4-3]\simeq \tau^4.\]
As $\U$ is $4$-cluster-tilting, it is $\nu_4$-equivariant by Proposition~\ref{nud}, and it follows that
\[
  0 = \Hom_{\dm{kQ}}\left(\left(\nu\circ[1])^3\circ\nu_4\right)(X),X[3]\right) \simeq \Hom_{\dm{kQ}}\left(\tau^4(X),X\right)
\]
implying that $4 > n-2$, that is, $n < 6$.
On the other hand, the condition \eqref{Dndiv} for $d=4$ becomes $5\mid2(2n-3)$ or, equivalently, $5\mid (n+1)$. Hence, $n=4$.
\end{proof}

To conclude the proof of Theorem~\ref{ade}, it remains only to show that $T(kQ)$ is $4$-representation-finite when $Q$ is of Dynkin type $D_4$.
To this end, consider the quiver $Q$ given by Figure~\ref{D4}, and let
\[\U=\add\{\tau^{5l}(P_1\oplus P_4) \mid l\in\Z\} \subset \dm{kQ}. \]
\begin{figure}[htb]
  \centerline{
  \xymatrix@=1.2em{
      1\ar[r] & 2\ar[r]\ar[d] &4 \\
      &  3
   }
  }
  \caption{Quiver $Q$ of Dynkin type $D_4$.}
  \label{D4}
\end{figure}
By Lemma~\ref{isom}\eqref{isom2}, $\nu\circ[1]\simeq \tau^{1-h} = \tau^{-5}$, so $\U$ is $(\nu\circ[1])$-equivariant.
Moreover, by Proposition~\ref{cydim}, $kQ$ is $(2/3)$-Calabi--Yau, and hence $[1]\simeq \tau^{-3}$ on $\dm{kQ}$.
It is now straightforward to verify that $\U$ is a $4$-cluster-tilting subcategory of $\dm{kQ}$ (Figure~\ref{ar-derd4}
), thus concluding the proof of Theorem~\ref{ade}.

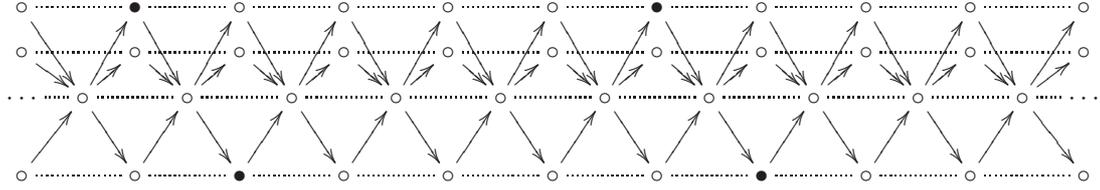
\begin{figure}[h]
 \[
  \xymatrix@R=.23cm@C=.3cm{
\circ\ar[ddr]\ar@{.}[rr] & & \bullet\ar[ddr]\ar@{.}[rr] & & \circ\ar[ddr]\ar@{.}[rr] & & \circ\ar[ddr]\ar@{.}[rr] & & \circ\ar[ddr]\ar@{.}[rr] & & \circ\ar[ddr]\ar@{.}[rr] & & \bullet\ar[ddr]\ar@{.}[rr] & & \circ\ar[ddr]\ar@{.}[rr] & & \circ\ar@{.}[rr]\ar[ddr]  & & \circ\ar@{.}[rr]\ar[ddr]  & & \circ\\
    \circ\ar[dr]\ar@{.}[rr] & & \circ\ar[dr]\ar@{.}[rr] & & \circ\ar[dr]\ar@{.}[rr] & & \circ\ar[dr]\ar@{.}[rr] & & \circ\ar[dr]\ar@{.}[rr] & & \circ\ar[dr]\ar@{.}[rr] & & \circ\ar[dr]\ar@{.}[rr] & & \circ\ar[dr]\ar@{.}[rr] & & \circ\ar@{.}[rr]\ar[dr]  & & \circ\ar@{.}[rr]\ar[dr]  & & \circ\\
    \cdots\ar@{.}[r]&\circ\ar[ddr]\ar@{.}[rr]\ar[uur]\ar[ur] & & \circ\ar[ddr]\ar@{.}[rr]\ar[uur]\ar[ur] & & \circ\ar[ddr]\ar@{.}[rr]\ar[uur]\ar[ur] & & \circ\ar[ddr]\ar@{.}[rr]\ar[uur]\ar[ur] & & \circ\ar[ddr]\ar@{.}[rr]\ar[uur]\ar[ur] & & \circ\ar[ddr]\ar@{.}[rr]\ar[uur]\ar[ur] & & \circ\ar[ddr]\ar@{.}[rr]\ar[uur]\ar[ur] & & \circ\ar[ddr]\ar[uur]\ar[ur]\ar@{.}[rr] & & \circ\ar[uur]\ar[ur]\ar[ddr]\ar@{.}[rr] & & \circ\ar[uur]\ar[ur]\ar[ddr]\ar@{.}[r]&\cdots\\
    \\
    \circ\ar[uur]\ar@{.}[rr] & & \circ\ar[uur]\ar@{.}[rr] & & \bullet\ar[uur]\ar@{.}[rr] & & \circ\ar[uur]\ar@{.}[rr] & & \circ\ar[uur]\ar@{.}[rr] & & \circ\ar[uur]\ar@{.}[rr] & & \circ\ar[uur]\ar@{.}[rr] & & \bullet\ar[uur]\ar@{.}[rr] & & \circ\ar@{.}[rr]\ar[ruu] & & \circ\ar@{.}[rr]\ar[ruu] & & \circ
  }
  \]
  \caption{The Auslander--Reiten quiver of $\dm{kQ}$ for $Q$ of type $D_4$, with the $4$-cluster-tilting subcategory $\U$ indicated in black.}
  \label{ar-derd4}
\end{figure}

\section{Proof of Corollary~\ref{rfsymm}}

Let $d$, $k$ and $A$ be as in Corollary~\ref{rfsymm}. 
From the classification of representation-finite symmetric $k$-algebras (\cite{gr79,riedtmann80a,riedtmann80b,riedtmann83a}, \cite{hw83,waschbusch81,waschbusch83}; see also \cite{skowronski06}), it follows that at least one of the following conditions holds.
\begin{enumerate}
\item $A$ is stably equivalent to a symmetric Nakayama algebra;
\item $A$ is a trivial extension of a tilted algebra of Dynkin type $D$ or $E$; or
\item $A$ has tree class $D_{3m}$ for some $m\ge2$. 
\end{enumerate}

In the first case, Theorem~\ref{nakayama} implies that $A$ is of the type given in Corollary~\ref{rfsymm}\eqref{rfsymm1}.
In the second case, $A$ is stably equivalent to $T(kQ)$, where $Q$ is the corresponding Dynkin quiver, whence it follows from Theorem~\ref{ade} that $Q$ is of type $D_4$ and $d=4$.

If $A$ has tree class $D_{3m}$, then it is socle equivalent (see \cite[Sec.~2.4]{skowronski06}) to an algebra of the form $A' = k\widehat{B}/\phi$, where $B$ is tilted of type $D_{3m}$ and $\phi\in\Aut_k(\widehat{B})$ is an admissible automorphism of the repetitive category $\widehat{B}$ of $B$, such that $\phi^3 = \nu_{\widehat{B}}$ in the outer automorphism group \cite[Sec~2.6,~3.5-3.6]{skowronski06}. Thus $\dm{B}\simeq\dm{kQ}$ for $Q$ of type $D_{3m}$, and $(\phi^*)^3 \simeq \nu_{\widehat{B}}^* \simeq \nu\circ[1]$.

Let $\ind A$ and $\ind A'$ be sets of representatives for the isomorphism classes of indecomposable modules of $A$ and $A'$, respectively. 
As the algebras $A$ and $A'$ are socle equivalent, their Auslander--Reiten quivers are isomorphic as translation quivers. 
This isomorphism induces a bijection $\Phi:\ind A \to \ind A'$ satisfying
$\sthom_A(X,Y) \simeq \sthom_{A'}(\Phi(X),\Phi(Y))$ for all non-projective $X,Y\in\ind A$ (see the argument preceding Proposition~3.6 in \cite{ckl15}). Moreover, by Proposition~\ref{omega-phi}, $\Omega_{A'}(\Phi(X)) \simeq \Phi(\Omega_A(X))$. 
Now, for each non-negative integer $i$, 
\begin{equation} \label{extiso}
\begin{aligned}
  \Ext^i_A(X,Y) &\simeq \sthom_A(\Omega^i(X),Y) \simeq \sthom_{A'}(\Phi(\Omega^i(X)),\Phi(Y)) \\
  &\simeq \sthom_{A'}(\Omega^i(\Phi(X)),\Phi(Y)) \simeq   \Ext^i_{A'}(\Phi(X),\Phi(Y)) \,.
\end{aligned}
\end{equation}
In particular, if $\U$ is a $d$-cluster-tilting subcategory of $\mod A$, then $\add\Phi(\U\cap\ind A)$ is a $d$-cluster-tilting subcategory of $\mod A'$. But then, by \cite[Theorem~2.3(a)]{di20}, $\dm{B}\simeq\dm{kQ}$ contains a $\phi^*$-equivariant, and thus $(\nu\circ[1])$-equivariant, $d$-cluster-tilting subcategory. This is impossible by Theorem~\ref{ade}. It follows that $A$ is not $d$-representation-finite.

\begin{rmk}
  In spite of the isomorphism \eqref{extiso}, the algebras $A$ and $A'$ in case (c) above are in general not stably equivalent. Indeed, if $A$ is a \emph{non-standard} algebra then, by Asashiba's classification \cite{asashiba99}, it is not stably equivalent to $A'$.
\end{rmk}

\section{Examples of $d$-cluster-tilting modules}

Here, we give examples of $d$-cluster-tilting modules for $T(kQ)$ in each of the cases listed in Theorem~\ref{ade}.
The proofs of the results in this section are straightforward verifications.

For $n\ge2$, let $Q_n$ be the following quiver:
\[
Q_n :\quad
\xymatrix@=1.2em{
  1\ar[r] & 2\ar[r] &\cdots\ar[r]& n
}
\]
We denote by $S_i$ the simple $kQ_n$-module supported at the vertex $i$, and by $P_i$ its projective cover.

\begin{prop} \label{CTA}
  \begin{enumerate}
  \item
    The module $P_n\oplus T(kQ_n)$ is a $(2n-1)$-cluster-tilting module of $T(kQ_n)$.
  \item
    The module $P_1\oplus P_2\oplus S_2\oplus T(kQ_3)$ is a $2$-cluster-tilting module of $T(kQ_3)$ (Figure~\ref{a3}).
  \item
    The module\;
    $\displaystyle T(kQ_6) \oplus \bigoplus_{i=0}^3 \Omega_{T(kQ_6)}^{3i}(P_5\oplus P_6)$\;
    is a $2$-cluster-tilting module of $T(kQ_6)$ (Figure~\ref{a6}).
  \end{enumerate}
\end{prop}

Let $Q$ be the quiver of type $D_4$ given in Figure~\ref{D4}.

\begin{prop}\label{CTD}
  The module $P_1\oplus P_3 \oplus T(kQ)$ is a $4$-cluster-tilting module of $T(kQ)$ (Figure~\ref{d4}). 
\end{prop}


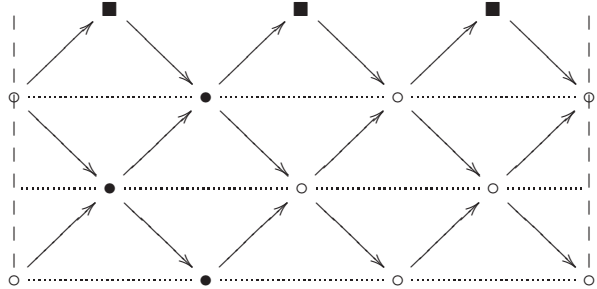
\begin{figure}[h]
  \[
  \xymatrix{
    \ar@{--}[ddd]& {}_{\blacksquare}\ar[dr] && {}_\blacksquare\ar[dr] && {}_\blacksquare\ar[dr] & \ar@{--}[ddd]\\
    \circ\ar[dr]\ar[ur]\ar@{.}[rr] && \bullet\ar[dr]\ar[ur]\ar@{.}[rr] &&
    \circ\ar[dr]\ar[ur]\ar@{.}[rr] &&\circ \\
    \ar@{.}[r]& \bullet\ar[ur]\ar[dr]\ar@{.}[rr] && \circ\ar[ur]\ar[dr]\ar@{.}[rr] &&
    \circ\ar[ur]\ar[dr]\ar@{.}[r] & \\
    \circ\ar[ur]\ar@{.}[rr] && \bullet\ar[ur]\ar@{.}[rr] && \circ\ar[ur]\ar@{.}[rr] && \circ \\
  }
  \]
  \caption{A $2$-cluster-tilting module of $T(kQ_3)$.}
  \label{a3}
\end{figure}


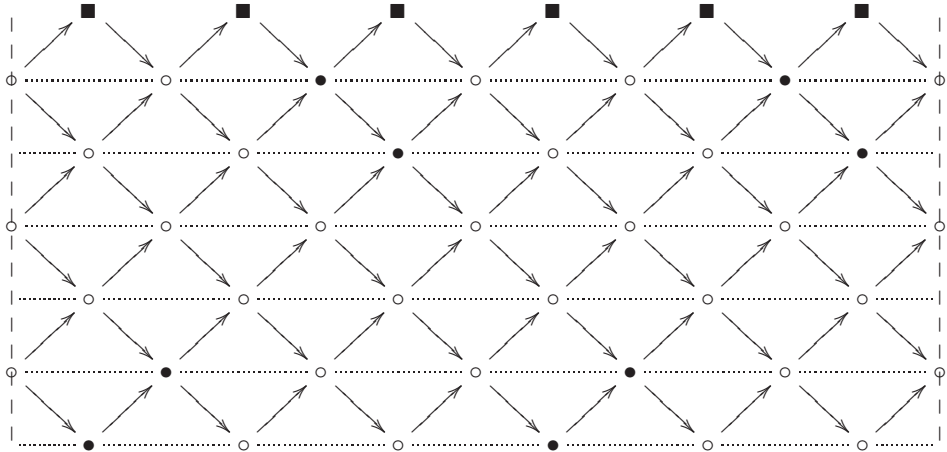
\begin{figure}[h]
  \[
  \xymatrix@R=.6cm@C=.6cm{
    \ar@{--}[dddddd]& {}_\blacksquare\ar[dr] && {}_\blacksquare\ar[dr] && {}_\blacksquare\ar[dr] && {}_\blacksquare\ar[dr] && {}_\blacksquare\ar[dr] &&  {}_\blacksquare\ar[dr] & \ar@{--}[dddddd] \\
    \circ\ar[ur]\ar[dr]\ar@{.}[rr] && \circ\ar[ur]\ar[dr]\ar@{.}[rr] &&
    \bullet\ar[ur]\ar[dr]\ar@{.}[rr] && \circ\ar[ur]\ar[dr]\ar@{.}[rr] && \circ\ar[ur]\ar[dr]\ar@{.}[rr] && \bullet\ar[dr]\ar[ur]\ar@{.}[rr] &&\circ \\
    \ar@{.}[r]& \circ\ar[ur]\ar[dr]\ar@{.}[rr] && \circ\ar[ur]\ar[dr]\ar@{.}[rr] && \bullet\ar[ur]\ar[dr]\ar@{.}[rr] && \circ\ar[ur]\ar[dr]\ar@{.}[rr] && \circ\ar[ur]\ar[dr]\ar@{.}[rr] && \bullet\ar[dr]\ar[ur]\ar@{.}[r] &\\
    \circ\ar[dr]\ar[ur]\ar@{.}[rr] && \circ\ar[dr]\ar[ur]\ar@{.}[rr] && \circ\ar[dr]\ar[ur]\ar@{.}[rr] && \circ\ar[dr]\ar[ur]\ar@{.}[rr] && \circ\ar[dr]\ar[ur]\ar@{.}[rr] && \circ\ar[dr]\ar[ur]\ar@{.}[rr] &&\circ \\
    \ar@{.}[r]& \circ\ar[ur]\ar[dr]\ar@{.}[rr] && \circ\ar[ur]\ar[dr]\ar@{.}[rr] &&
    \circ\ar[ur]\ar[dr]\ar@{.}[rr] && \circ\ar[ur]\ar[dr]\ar@{.}[rr] && \circ\ar[ur]\ar[dr]\ar@{.}[rr] && \circ\ar[dr]\ar[ur]\ar@{.}[r] &\\
    \circ\ar[dr]\ar[ur]\ar@{.}[rr] && \bullet\ar[dr]\ar[ur]\ar@{.}[rr] && \circ\ar[dr]\ar[ur]\ar@{.}[rr] && \circ\ar[dr]\ar[ur]\ar@{.}[rr] && \bullet\ar[dr]\ar[ur]\ar@{.}[rr] && \circ\ar[dr]\ar[ur]\ar@{.}[rr] &&\circ \\
      \ar@{.}[r]& \bullet\ar[ur]\ar@{.}[rr] && \circ\ar[ur]\ar@{.}[rr] &&
    \circ\ar[ur]\ar@{.}[rr] && \bullet\ar[ur]\ar@{.}[rr] && \circ\ar[ur]\ar@{.}[rr] && \circ\ar[ur]\ar@{.}[r] &\\
}
  \]
  \caption{A $2$-cluster-tilting module of $T(kQ_6)$.}
  \label{a6}
\end{figure}


\begin{figure}[h]
 \[
 \xymatrix@R=.4cm{ 
\ar@{--}[dddddd]&&&&&&& {}_\blacksquare\ar[dr]&&& \ar@{--}[dddddd]   \\
\circ\ar[ddr]\ar@{.}[rr] & & \bullet\ar[ddr]\ar@{.}[rr] & & \circ\ar[ddr]\ar@{.}[rr] & & \circ\ar[ddr]\ar@/^.2pc/@{.}[rr]\ar[ur] & {}_\blacksquare\ar[dr] & \circ\ar[ddr]\ar@{.}[rr] & & \circ \\
    \circ\ar[dr]\ar@{.}[rr] & & \circ\ar[dr]\ar@{.}[rr] & & \circ\ar[dr]\ar@{.}[rr] & & \circ\ar[dr]\ar@{.}[rr]\ar[ur] & & \circ\ar[dr]\ar@{.}[rr] & & \circ \\
\ar@{.}[r]&\circ\ar[ddr]\ar@{.}[rr]\ar[uur]\ar[ur] & & \circ\ar[ddr]\ar@{.}[rr]\ar[uur]\ar[ur] & & \circ\ar[ddr]\ar@{.}[rr]\ar[uur]\ar[ur] & & \circ\ar[ddr]\ar@{.}[rr]\ar[uur]\ar[ur] & & \circ\ar[ddr]\ar@{.}[r]\ar[uur]\ar[ur] &  \\
    \\
    \circ\ar[uur]\ar@{.}[rr]\ar[dr] & & \circ\ar[uur]\ar@{.}[rr]\ar[dr] & & \bullet\ar[uur]\ar@{.}[rr] & & \circ\ar[uur]\ar@{.}[rr] & & \circ\ar[uur]\ar@{.}[rr] & & \circ\\
& {}_\blacksquare\ar[ur] && {}_\blacksquare\ar[ur]&&&&&&&
  }
  \]
  \caption{A $4$-cluster-tilting module of $T(kQ)$, with $Q$ of type $D_4$.}
  \label{d4}
\end{figure}
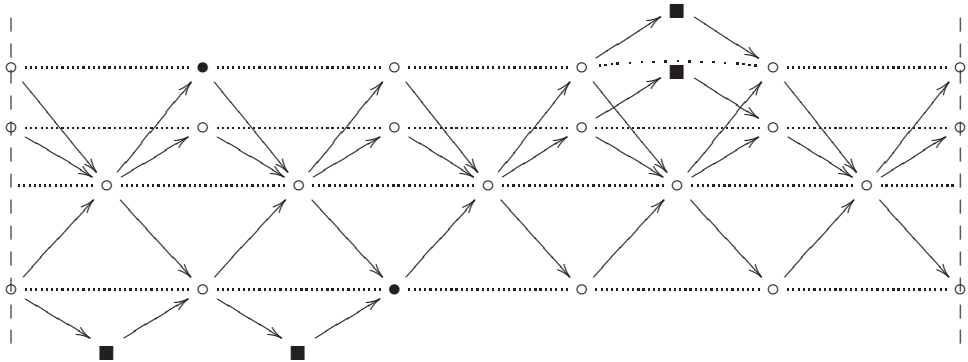

\section*{Acknowledgements}

We are grateful to the anonymous referee, whose comments helped us to improve the exposition and alerted us to an incompleteness in the proof of Corollary~\ref{rfsymm}.


\def\cprime{$'$}

\end{document}